\pdfoutput=1
\documentclass[12pt,letter]{amsart}
\usepackage{xcolor}
\usepackage[all,color]{xy}
\usepackage{amsmath,amssymb,amsthm,amscd,amsxtra,amsfonts,mathrsfs,graphicx,enumerate,bm,slashed,array,setspace,amsfonts,enumitem}
\input xy
\xyoption{all}
\newtheorem{Theorem}{Theorem}[section]
\newtheorem{Lemma}[Theorem]{Lemma}
\newtheorem{Proposition}[Theorem]{Proposition}
\newtheorem{Corollary}[Theorem]{Corollary}

\newtheorem{Remark}[Theorem]{Remark}
\newtheorem{Definition}[Theorem]{Definition}

\newtheorem*{Theorem A}{Theorem A}

\newcommand*{\overbar}[1]{\mkern 1.5mu\overline{\mkern-1.5mu#1\mkern-1.5mu}\mkern 1.5mu}
\advance\evensidemargin-.5in
\advance\oddsidemargin-.5in
\advance\textwidth1in

\begin{document}
\author{Charlie Beil}
 %\thanks{}
\address{Institute for Mathematics and Scientific Computing, University of Graz, Heinrichstrasse 36, 8010 Graz, Austria.}
 \email{charles.beil@uni-graz.at}
 \title[Noetherian criteria for dimer algebras]{Noetherian criteria for dimer algebras}
 \date{}

\begin{abstract}
Let $A$ be a nondegenerate dimer (or ghor) algebra on a torus, and let $Z$ be its center.
Using cyclic contractions, we show the following are equivalent: $A$ is noetherian; $Z$ is noetherian; $A$ is a noncommutative crepant resolution; each arrow of $A$ is contained in a perfect matching whose complement supports a simple module; and the vertex corner rings $e_iAe_i$ are pairwise isomorphic.
\end{abstract}

\maketitle

\section{Introduction}

In this article, all dimer quivers are on a real two-torus $T^2$, and all algebras are over an algebraically closed base field $k$.
A prominent class of noncommutative crepant resolutions (NCCRs) and Calabi-Yau algebras are cancellative dimer algebras.\footnote{For the definition of an NCCR and a Calabi-Yau algebra, see for example \cite{V,Br}.}
In fact, a dimer algebra is an NCCR and Calabi-Yau if and only if it is cancellative \cite{MR,D,Br}.
Cancellativity has been shown to be equivalent to certain combinatorial conditions on the dual graph of the quiver \cite{HV,IU,KS,G}.
The main objective of this article is to characterize cancellative dimer algebras in terms of noetherianity, vertex corner ring structure, and perfect matchings whose complements support simple modules.

Dimer algebras were introduced in string theory in 2005 to describe a class of quiver gauge theories \cite{HK,F-K}.
A \textit{dimer algebra} is a quiver algebra $A = kQ/I$ of a quiver $Q$ whose underlying graph $\overbar{Q}$ embeds in a torus $T^2$ (or more generally, a compact surface), such that each connected component of $T^2 \setminus \overbar{Q}$ is simply connected and bounded by an oriented cycle, called a unit cycle.
The relations of $A$ are given by the ideal
$$I := \left\langle p - q \ | \ \exists \ a \in Q_1 \text{ such that } pa \text{ and } qa \text{ are unit cycles} \right\rangle \subset kQ,$$
where $p$ and $q$ are paths.
Since $I$ is generated by certain differences of paths, we may refer to a path modulo $I$ as a path in the dimer algebra $A$.

A distinguishing property of dimer algebras is cancellativity.
Two paths $p,q$ in a dimer algebra $A$ are said to form a \textit{non-cancellative pair} if $p \not = q$, and there is a path $r \in A$ such that
$$rp = rq \not = 0 \ \ \text{ or } \ \ pr = qr \not = 0.$$
If such a pair exists, then $A$ and $Q$ are called \textit{non-cancellative}; otherwise they are \textit{cancellative}.
We will show that a dimer algebra is cancellative if and only if it is noetherian.

We will also characterize noetherianity for ghor algebras.
The \textit{ghor algebra} of a dimer quiver $Q$ on a torus is (isomorphic to) the quotient of the dimer algebra $A = kQ/I$ by non-cancellative pairs,\footnote{In \cite{B5}, we called ghor algebras `homotopy algebras' since their relations identify homotopic paths in the quiver if the surface is a torus.
However, in \cite{BB} we showed that on higher genus surfaces homologous cycles in the quiver are also identified, and so the prefix `homotopy' became less suitable.  
The word `ghor' is Klingon for surface.}
$$\Lambda = A/\langle p - q \ | \ p,q \text{ is a non-cancellative pair} \rangle.$$
Ghor algebras were introduced in \cite{B2}, and often have nicer algebraic and homological properties than do their dimer counterparts \cite{B5,B7,BB}.
Note that a dimer algebra coincides with its ghor algebra if and only if it is cancellative.\footnote{Using the definition of a ghor algebra given in \cite[Introduction]{B2}, the statement that a dimer algebra (on a torus) coincides with its ghor algebra if and only if it is cancellative is nontrivial and is shown in \cite[Theorem 4.31]{B2}.}

The primary tool we use to study noetherianity is the notion of a cyclic contraction, also introduced in \cite{B2}.
A \textit{cyclic contraction} is map of dimer algebras $\psi: kQ/I \to kQ'/I'$, such that $Q'$ is obtained from $Q$ by contracting a set of arrows to vertices whilst perserving the cycles of $Q$ in a suitable sense (see Section \ref{cyclic section}). 
Cyclic contractions enable non-cancellative and cancellative dimer algebras to be related to each other.
An example of a cyclic contraction is given in Figure \ref{example}.
In this example, $A = kQ/I$ is non-cancellative and $A' = kQ'/I'$ is cancellative: Let $a,b,c$ be the respective red, blue, and green arrows in $Q$, as shown in the figure.
Then the paths $ab, ba \in A$ form a non-cancellative pair since
$$ab \not = ba \ \ \ \text{ and } \ \ \ cab = cba \not = 0.$$

\begin{figure}
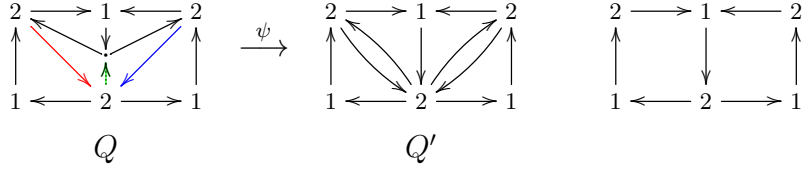

$$\begin{array}{ccccc}
\xy
(-12,6)*+{\text{\scriptsize{$2$}}}="1";(0,6)*+{\text{\scriptsize{$1$}}}="2";(12,6)*+{\text{\scriptsize{$2$}}}="3";
(0,0)*{\cdot}="7";
(-12,-6)*+{\text{\scriptsize{$1$}}}="4";(0,-6)*+{\text{\scriptsize{$2$}}}="5";(12,-6)*+{\text{\scriptsize{$1$}}}="6";
{\ar"6";"3"};{\ar@[blue]"3";"5"};{\ar^{}"7";"3"};{\ar^{}"3";"2"};
{\ar^{}"1";"2"};{\ar^{}"2";"7"};{\ar^{}"7";"1"};{\ar@[red]"1";"5"};{\ar^{}"5";"4"};{\ar^{}"4";"1"};{\ar^{}"5";"6"};
{\ar@[green]"5";"7"};{\ar@{..>}"5";"7"};
\endxy
& \stackrel{\psi}{\longrightarrow} &
\xy
(-12,-6)*+{\text{\scriptsize{$1$}}}="4";(0,-6)*+{\text{\scriptsize{$2$}}}="5";(12,-6)*+{\text{\scriptsize{$1$}}}="6";
(-12,6)*+{\text{\scriptsize{$2$}}}="1";(0,6)*+{\text{\scriptsize{$1$}}}="2";(12,6)*+{\text{\scriptsize{$2$}}}="3";
{\ar@/_.3pc/"1";"5"};{\ar@/_.3pc/"5";"1"};{\ar@/^.3pc/"5";"3"};{\ar@/^.3pc/"3";"5"};
{\ar^{}"1";"2"};{\ar^{}"2";"5"};{\ar^{}"5";"4"};{\ar^{}"4";"1"};{\ar^{}"5";"6"};{\ar^{}"6";"3"};{\ar^{}"3";"2"};
\endxy
& &
\xy
(-12,-6)*+{\text{\scriptsize{$1$}}}="4";(0,-6)*+{\text{\scriptsize{$2$}}}="5";(12,-6)*+{\text{\scriptsize{$1$}}}="6";
(-12,6)*+{\text{\scriptsize{$2$}}}="1";(0,6)*+{\text{\scriptsize{$1$}}}="2";(12,6)*+{\text{\scriptsize{$2$}}}="3";
{\ar^{}"1";"2"};{\ar^{}"2";"5"};{\ar^{}"5";"4"};{\ar^{}"4";"1"};{\ar^{}"5";"6"};{\ar^{}"6";"3"};{\ar^{}"3";"2"};
\endxy
\\
Q & \ \ \ & Q' & \ \ \ &
\end{array}$$
\caption{An example of a cyclic contraction.
Each quiver is drawn on a torus, and $\psi: A \to A'$ contracts the green arrow.
The arrows in the two unit 2-cycles in $Q'$ are redundant generators for $A'$ and so may be removed.}
\label{example}
\end{figure}

Our main theorem is the following.

\begin{Theorem} \label{big theorem}
(Theorems \ref{cool!} and \ref{iff}.)
Let $A = kQ/I$ be a nondegenerate dimer algebra on a torus with center $Z$, and let $\Lambda$ be the corresponding ghor algebra with center $R$.
The following are equivalent:
\begin{enumerate}
 \item $A$ is cancellative.
 \item $A$ is noetherian.
 \item $Z$ is noetherian.
 \item $A$ is a finitely generated $Z$-module.
 \item The vertex corner rings $e_iAe_i$ are pairwise isomorphic.
 \item Each vertex corner ring $e_iAe_i$ is isomorphic to $Z$.
 \item Each arrow annihilates a simple $A$-module of dimension vector $1^{Q_0}$.
 \item Each arrow is contained in a simple matching.
 \item The center $R$ of $\Lambda$ equals the cycle algebra $S$ of $A$.
 \item If $\psi: A \to A'$ is a cyclic contraction, then $\psi$ is trivial.
\end{enumerate}
Furthermore, these conditions are equivalent to each condition (2) -- (7) with $A$ and $Z$ replaced by $\Lambda$ and $R$.
\end{Theorem}

More precisely, in (9) we mean that $R$ is generated by the monomial images of all the cycles of $Q$ under the canonical quotient map $kQ \to \Lambda$.
The implications (1) $\Rightarrow$ (2) - (6) are well known (e.g., \cite{D,MR}, \cite[Proposition 4.28]{B2}); one of the main objectives of this article is to prove the converse implications, which are new and much more involved.
Some of the equivalences do not hold in general for dimer or ghor algebras on surfaces other than the torus, and are thus unexpected.
For example, the equivalences (2) $\Leftrightarrow $ (8) $\Leftrightarrow $ (6) do not hold on higher genus surfaces by \cite[Proposition 3.9, Lemma 4.3]{BB}.

\section{Preliminaries}

Given a quiver $Q$, we denote by $kQ$ the path algebra of $Q$, and by $Q_{\ell}$ the paths of length $\ell$.
The idempotent at vertex $i \in Q_0$ is denoted $e_i$, and the head and tail maps are denoted $\operatorname{h},\operatorname{t}: Q_1 \to Q_0$.
Multiplication of paths is read right to left, following the composition of maps.
By module we mean left module; and by infinitely generated module, we mean a module that is not finitely generated.
The `support' of a module $V$ over a quiver algebra $kQ/I$ is the largest subquiver $Q' \subseteq Q$ for which no vertex or arrow of $Q'$ annihilates $V$. 
Finally, we denote by $e_{ji} \in M_n(k)$ the $n \times n$ matrix with a $1$ in the $ji$-th slot and zeros elsewhere. 
%By monomial, we mean a non-constant monomial.

\subsection{Covering quivers and perfect matchings}

Let $A = kQ/I$ be a dimer algebra.
If $p,q$ are paths in $Q$ that are equal modulo $I$, then we will write $p \equiv q$; if $p,q$ are regarded as paths in $A$, then we will write $p = q$.

$\bullet$ Consider a covering map $\pi: \mathbb{R}^2 \rightarrow T^2$ such that for some $i \in Q_0$,
$$\pi\left(\mathbb{Z}^2 \right) = i.$$
Denote by $Q^+ := \pi^{-1}(Q) \subset \mathbb{R}^2$ the infinite covering quiver of $Q$.
For each path $p$ in $Q$, denote by $p^+$ the unique path in $Q^+$ with tail in the unit square $[0,1) \times [0,1) \subset \mathbb{R}^2$ satisfying $\pi(p^+) = p$.
If two paths $p^+$ and $q^+$ in $Q^+$ have coincident tails and coincident heads, then we denote the compact region they bound by $\mathcal{R}_{p,q}$ (or $\mathcal{R}_{p+I,q+I}$ if the representatives $p$ and $q$ are clear from context.)

$\bullet$ We will consider the following sets of arrows:
 \begin{itemize}
  \item[--] A \textit{perfect matching} $x \subset Q_1$ is a set of arrows such that each unit cycle contains precisely one arrow in $x$.
  \item[--] A \textit{simple matching} $x \subset Q_1$ is a perfect matching such that $Q \setminus x$ supports a simple $A$-module of dimension $1^{Q_0}$ (that is, the subquiver of $Q$ with arrow set $Q_1 \setminus x$ contains a cycle that passes through each vertex of $Q$).
 \end{itemize}
Denote by $\mathcal{P}$ and $\mathcal{S}$ the set of perfect and simple matchings of $Q$, respectively.
 $A$ and $Q$ are said to be \textit{nondegenerate} if each arrow is contained in a perfect matching.

$\bullet$ By a \textit{cyclic subpath} of a path $p$, we mean a subpath of $p$ that is a nontrivial cycle.
We will consider the following sets of cycles in $A$:
\begin{itemize}
 \item[--] Denote by $\mathcal{C}$ the set of cycles in $A$ (i.e., cycles in $Q$ modulo $I$).
 \item[--] For $u \in \mathbb{Z}^2$, denote by $\mathcal{C}^u$ the set of cycles $p \in \mathcal{C}$ such that
$$\operatorname{h}(p^+) = \operatorname{t}(p^+) + u \in Q_0^+.$$
 \item[--] For $i \in Q_0$, denote by $\mathcal{C}_i$ the set of cycles in the vertex corner ring $e_iAe_i$.
 \item[--] Denote by $\hat{\mathcal{C}}$ the set of cycles $p \in \mathcal{C}$ such that the lift of each cyclic permutation of each representative of $p$ does not have a cyclic subpath.
\end{itemize}
We decorate $\mathcal{C}$ so as to specify a set of cycles; e.g., $\hat{\mathcal{C}}^u_i := \hat{\mathcal{C}} \cap \mathcal{C}^u \cap \mathcal{C}_i$.
Note that $\hat{\mathcal{C}}^0 = Q_0$.

It is well-known that if $\sigma_i, \sigma'_i$ are two unit cycles at $i \in Q_0$, then $\sigma_i = \sigma'_i$ in $A$; we will denote by $\sigma_i \in A$ the unique unit cycle at $i$.
Furthermore, the sum $\sum_{i \in Q_0}\sigma_i$ is in the center of $A$.

\subsection{Matrix ring homomorphisms}

Let $A = kQ/I$ be a dimer algebra.
By \cite[Lemma 2.1]{B2}, there are algebra homomorphisms from $A$ to the matrix rings
\begin{equation} \label{te}
\tau: A \to M_{|Q_0|}\left(k[\mathcal{S}]\right) \ \ \ \ \text{ and } \ \ \ \ \eta: A \to M_{|Q_0|}\left(k[\mathcal{P}] \right)
\end{equation}
defined on $i \in Q_0$ and $a \in Q_1$ by
\begin{align*}
\tau(e_i) = e_{ii}, \ \ \ \ \ \ & \eta(e_i) = e_{ii},\\
\tau(a) = e_{\operatorname{h}(a),\operatorname{t}(a)} \prod_{\substack{x \in \mathcal{S}: \\ x \ni a}} x, \ \ \ \ \ \ & \eta(a) = e_{\operatorname{h}(a),\operatorname{t}(a)} \prod_{\substack{x \in \mathcal{P}: \\ x \ni a}} x,
\end{align*}
and extended multiplicatively and $k$-linearly to $A$.

For each $i,j \in Q_0$, denote by
$$\bar{\tau}: e_jAe_i \to k[\mathcal{S}] \ \ \ \ \text{ and } \ \ \ \ \bar{\eta}: e_jAe_i \to k[ \mathcal{P}]$$
the respective $k$-linear maps defined on $p \in e_jAe_i$ by
$$\tau(p) = \bar{\tau}(p)e_{ji} \ \ \ \ \text{ and } \ \ \ \ \eta(p) = \bar{\eta}(p)e_{ji}.$$
Denote by $\sigma$ the product of all the variables,
$$\sigma := \bar{\tau}(\sigma_i) = \prod_{x \in \mathcal{S}} x \ \ \ \ \text{ or } \ \ \ \ \sigma := \bar{\eta}(\sigma_i) = \prod_{x \in \mathcal{P}}x.$$

We will make use of the following results from \cite{B2}.
See Figure \ref{prop2.1} for reference.

\begin{Proposition} \label{Bastuff}
Let $A$ be a nondegenerate dimer algebra.
Given a path $p$, denote by $\overbar{p}$ either $\bar{\eta}(p)$ or $\bar{\tau}(p)$.\\
$\bullet$ First suppose $p,q \in e_jAe_i$ are distinct paths satisfying
\begin{equation*} \label{p+}
\operatorname{t}(p^+) = \operatorname{t}(q^+) \ \ \text{ and } \ \ \operatorname{h}(p^+) = \operatorname{h}(q^+).
\end{equation*}
Then:
\begin{enumerate}
 \item There is an $m \in \mathbb{Z}$ such that $\overbar{p} = \overbar{q} \sigma^m$.
 \item $p,q$ is a non-cancellative pair if and only if $\bar{\eta}(p) = \bar{\eta}(q)$.
\end{enumerate}
$\bullet$ Now suppose either $A$ is cancellative, or there is some $u \in \mathbb{Z}^2 \setminus 0$ such that $\hat{\mathcal{C}}_i^u \not = \emptyset$ for each $i \in Q_0$.
\begin{enumerate}[resume]
% \item If $p \in \mathcal{C}^0$, then $\overbar{p} = \sigma^m$ for some $m \geq 0$.
 \item Let $p \in \mathcal{C}$.
 Then $p \in \hat{\mathcal{C}}$ if and only if $\sigma \nmid \overbar{p}$.
 \item If $p,q \in \hat{\mathcal{C}}^u$, then $\overbar{p} = \overbar{q}$.
\end{enumerate}
$\bullet$ Finally, suppose $A$ is cancellative.
Then:
\begin{enumerate}[resume]
 \item For each $u \in \mathbb{Z}^2$ and $i \in Q_0$, the set $\hat{\mathcal{C}}^u_i$ is nonempty.
 \item Let $p \in \mathcal{C}^u$.
Then $u = 0$ if and only if there is some $\ell \geq 0$ such that $\overbar{p} = \sigma^{\ell}$.
 \item For each $i,j \in Q_0$, we have
$$e_iAe_i = Ze_i \cong Z \cong \bar{\tau}\left(e_iAe_i \right) = \bar{\tau}\left(e_jAe_j\right).$$
 \item $A$ is a finitely generated $Z$-module, and $Z \cong \bar{\tau}(e_iAe_i)$ is a finitely generated $k$-algebra, generated by $\sigma$ and a finite set of monomials in $k[\mathcal{S}]$ not divisible by $\sigma$.
\end{enumerate}
\end{Proposition}

\begin{proof}
The claims are respectively (1), (2): \cite[Lemma 4.3]{B2}; (3), (4): \cite[Lemma 4.8, Proposition 4.21]{B2}; (5): \cite[Proposition 4.11]{B2}; (6): \cite[Lemma 4.29]{B2}; and (7), (8): \cite[Propositions 4.28, 5.14, Theorem 5.9.3]{B2}. %; and (8): \cite[Proposition 5.14]{B2}.
\end{proof}

\begin{figure}
$$\begin{array}{cc}
\xy 0;/r.68pc/:
(-6,0)*+{\text{\footnotesize{$i$}}}="1";(-2,3)*{}="2";(2,3)*{}="4";
(6,0)*+{\text{\footnotesize{$j$}}}="8";
(-2,-3)*{}="5";(2,-3)*{}="7";
{\ar@{-}@/^/^{}"1";"2"};
{\ar@{-}@/^.1pc/^p"2";"4"};
{\ar@/^/"4";"8"};
{\ar@{-}@/_/_{}"1";"5"};
{\ar@{-}@/_.1pc/_{q}"5";"7"};
{\ar@/_/_{}"7";"8"};
\endxy
&
\xy 0;/r.8pc/:
(-6,-2)*+{\text{\footnotesize{$i$}}}="1";(-6,3)*+{\text{\footnotesize{$j$}}}="2";(6,3)*+{\text{\footnotesize{$j$}}}="4";
(6,-2)*+{\text{\footnotesize{$i$}}}="8";
{\ar@/^.8pc/^{p_1}"1";"2"};
{\ar@/^.8pc/^{p_3}"2";"1"};
{\ar^{p_2}"2";"4"};
{\ar@/^.8pc/^{p_3}"4";"8"};
{\ar@/^.8pc/^{p_1}"8";"4"};
\endxy
\\
(i) & (ii)
\end{array}$$
\caption{Examples for Proposition \ref{Bastuff}.
All paths shown are paths of positive length in the cover $Q^+$ (with the superscripts $^+$ omitted).
In (i), the paths $p, q \in e_jAe_i$ satisfy $\operatorname{t}(p^+) = \operatorname{t}(q^+)$ and $\operatorname{h}(p^+) = \operatorname{h}(q^+)$.
In (ii), the path $p = p_3p_2p_1$ is in $\mathcal{C}_i \setminus \hat{\mathcal{C}}_i$, since the cyclic permutation $(p_1p_3p_2)^+$ of $p^+$ has a nontrivial cyclic subpath.
Furthermore, $p_3p_1$ is in $\mathcal{C}^0$, and therefore $\overbar{p_3p_1} = \sigma^{\ell}$ for some $\ell \geq 1$.}  
\label{prop2.1}
\end{figure}

\subsection{Cyclic contractions} \label{cyclic section}

Let $A = kQ/I$ and $A' = kQ'/I'$ be dimer algebras, and suppose $Q'$ is obtained from $Q$ by contracting a set of arrows $Q_1^* \subset Q_1$ to vertices.
This operation defines a $k$-linear map of path algebras
$$\psi: kQ \to kQ'.$$
If $\psi(I) \subseteq I'$, then $\psi$ induces a $k$-linear map of dimer algebras, $\psi: A \to A'$, called a \textit{contraction}.
Consider the algebra homomorphism
$$\tau: A' \to M_{|Q'_0|}(k[\mathcal{S}'])$$
defined in (\ref{te}).
We call a contraction $\psi: A \to A'$ \textit{cyclic} if $A'$ is cancellative and
\begin{equation*} \label{cycle algebra}
S := k \left[ \cup_{i \in Q_0} \bar{\tau}\psi(e_iAe_i) \right] = k \left[ \cup_{i \in Q'_0} \bar{\tau}(e_iA'e_i) \right].
\end{equation*}

The algebra $S$, called the \textit{cycle algebra} of $A$, is independent of the choice of cyclic contraction $\psi$ \cite[Theorem 3.14]{B3}.
Furthermore, every nondegenerate dimer algebra admits a cyclic contraction \cite[Theorem 1.1]{B1}.

Finally, we remark that if $\psi: A \to A'$ is a cyclic contraction, then the ghor algebra of $\Lambda'$ of $A'$ is simply $A'$ itself since $A'$ is cancellative,
$$\Lambda' = A'/\left\langle p - q \ | \ p,q \text{ a non-cancellative pair} \right\rangle = A'.$$
Furthermore, the cycle algebra $S$ and the algebra generated by the intersection
$$R = k\left[ \cap_{i \in Q_0} \bar{\tau}\psi(e_iAe_i) \right]$$
are the centers of the ghor algebras of $A'$ and $A$ respectively \cite[Theorem 5.9.3]{B2}.

\section{Proof of main theorem}

Throughout, let $A$ be a nondegenerate dimer algebra.

\begin{Definition} \rm{
We say paths $p,q \in Q_{\geq 0}$ are a non-cancellative pair if they are representatives of a non-cancellative pair $p+I,q+I \in A$.
We say a non-cancellative pair $p,q \in Q_{\geq 0}$ is minimal if, given any non-cancellative pair $s,t \in Q_{\geq 0}$ satisfying
$$\mathcal{R}_{s,t} \subseteq \mathcal{R}_{p,q},$$
we have $\{s,t\}  = \{p,q\}$.
Let $r \in Q_{\geq 0}$ be a path for which $rp \equiv rq$ (resp.\ $pr \equiv qr$).
We say $r$ is \textit{minimal} if there is no proper subpath $r'$ of $r$ satisfying $r'p \equiv r'q$ ($pr' \equiv qr'$).
}\end{Definition}

\begin{Lemma} \label{p-q in I}
Let $p,q \in Q_{\geq 0}$ be a minimal non-cancellative pair, and let $r$ be a minimal path satisfying $rp \equiv rq$ (resp.\ $pr \equiv qr$).
If the rightmost (leftmost) arrow subpath of $r^+$ lies in $\mathcal{R}_{p,q}$, then $r^+$ lies wholly in $\mathcal{R}_{p,q}$, and only meets $p^+$ or $q^+$ at its tail (head).
\end{Lemma}

\begin{proof}
Set $\sigma := \prod_{x \in \mathcal{P}} x$, and for a path $s$, set $\overbar{s} := \bar{\eta}(s)$.

Suppose that $rp \equiv rq$, and the rightmost arrow subpath of $r^+$ lies in $\mathcal{R}_{p,q}$.
Assume to the contrary that $r^+$ meets $p^+$ at a vertex other than its tail; see Figure \ref{interiorlemma2}.
Then $p$ and $r$ factor into paths 
$$p = p_2p_1 \ \ \ \text{ and } \ \ \ r = r_3r_2r_1,$$
where $p_2, r_1,r_3 \in Q_{\geq 0}$ are paths (of unspecified length), $r_2 \in Q_1$ is an arrow, and $\operatorname{h}(p_1) = \operatorname{h}(r_2)$.

Since $r_2$ is an arrow, the cycle $r_2r_1p_2 \in \mathcal{C}^0$ is nontrivial.
Therefore there is some $m \geq 1$ such that $\overbar{r_2r_1p_2} = \sigma^m$, by Proposition \ref{Bastuff}.1.
In particular, the paths $p_1 \sigma_i^m$ and $r_2r_1q$ satisfy
$$\overbar{p}_1 \sigma^m = \overbar{r_2r_1p_2} \, \overbar{p}_1 = \overbar{r_2r_1p} = \overbar{r_2r_1} \, \overbar{p} \stackrel{\textsc{(i)}}{=} \overbar{r_2r_1} \, \overbar{q} = \overbar{r_2r_1q},$$
where (\textsc{i}) holds by Proposition \ref{Bastuff}.2.
Consequently, either $p_1 \sigma_{\operatorname{t}(p)}^m$ and $r_2r_1q$ are equal modulo $I$, or they form a non-cancellative pair, by Proposition \ref{Bastuff}.2.

(i) First suppose  $p_1 \sigma_{\operatorname{t}(p)}^m$ and $r_2r_1q$  form a non-cancellative pair.
By choosing a representative of the unit cycle $\sigma_{\operatorname{t}(p)}^+$ that lies in $\mathcal{R}_{p_1,r_2r_1q}$, there is proper containment
$$\mathcal{R}_{p_1 \sigma_{\operatorname{t}(p)}^m, r_2r_1q} \subset \mathcal{R}_{p,q}.$$
But then the non-cancellative pair $p,q$ is not minimal, contrary to assumption.

(ii) So suppose  $p_1 \sigma_{\operatorname{t}(p)}^m \equiv r_2r_1q$.
Since $r_2$ is an arrow, there is a path $s$ for which $(sr_2)^+$ is a unit cycle contained in $\mathcal{R}_{p_1,r_2r_1q}$.
By Proposition \ref{Bastuff}.1, there is an $n \in \mathbb{Z}$ such that
\begin{equation} \label{was a merry fellow}
\overbar{r_1p_2} = \overbar{s}\sigma^n.
\end{equation}
Since $s$ is a subpath of a unit cycle, we have $n \geq 0$.
Whence,
\begin{equation} \label{hey ho howdy}
\overbar{s\sigma_{\operatorname{t}(s)}^n p_1} = \overbar{r_1p_2p_1} = \overbar{r_1p} \stackrel{\textsc{(i)}}{=} \overbar{r_1q},
\end{equation}
where (\textsc{i}) holds by Proposition \ref{Bastuff}.2.

Now (\ref{was a merry fellow}) implies that either $r_1p_2$ and $s \sigma_{\operatorname{t}(s)}^n$ are equal modulo $I$, or they form a non-cancellative pair, by Proposition \ref{Bastuff}.2.
Since $p,q$ is minimal and $s^+$ lies in $\mathcal{R}_{p_1,r_2r_1q}$, we have
\begin{equation} \label{snort}
r_1p_2 \equiv s \sigma_{\operatorname{t}(s)}^n.
\end{equation}
Thus
\begin{equation}\label{dozy dozy}
s\sigma_{\operatorname{t}(s)}^n p_1 \stackrel{\textsc{(i)}}{\equiv} r_1p_2p_1 \equiv r_1p \stackrel{\textsc{(ii)}}{\not \equiv} r_1q,
\end{equation}
where (\textsc{i}) holds by (\ref{snort}); and (\textsc{ii}) holds since $r$ has minimal length and $r_2$ is a nontrivial path.
Therefore (\ref{hey ho howdy}) and (\ref{dozy dozy}) together imply that the paths $s\sigma_{\operatorname{t}(s)}^n p_1$ and $r_1q$ form a non-cancellative pair, by Proposition \ref{Bastuff}.2.
But then $p,q$ is not minimal, again contrary to assumption.
\end{proof}

\begin{figure}
$$\xy 0;/r.8pc/:
(-6,0)*{\cdot}="1";(-2,3)*{}="2";(2,3)*{\cdot}="4";
(6,0)*{\cdot}="8";
(-2,-3)*{}="5";(2,-3)*{}="7";
(2,6)*{\cdot}="9";
(2,0)*{\cdot}="10";
{\ar@[green]_{r_3}"4";"9"};{\ar@{..>}"4";"9"};
{\ar@{-}@/^/^{}@[red]"1";"2"};{\ar@/^.1pc/^{p_1}@[red]"2";"4"};{\ar@/^/^{p_2}@[red]"4";"8"};
{\ar@{-}@/_/_{}@[blue]"1";"5"};{\ar@{-}@/_.1pc/_{q}@[blue]"5";"7"};{\ar@/_/_{}@[blue]"7";"8"};
{\ar@[green]^{r_1}"8";"10"};{\ar@{..>}"8";"10"};{\ar@[green]_{r_2}"10";"4"};{\ar@{..>}"10";"4"};
{\ar@/_1.7pc/_s"4";"10"};
\endxy$$
\caption{Setup for Lemma \ref{p-q in I}.  
Here, $p_1,p_2,q,r_1,r_3,s \in Q_{\geq 0}$ are paths (of unspecified length); $r_2 \in Q_1$ is an arrow; $r_2s$ is a unit cycle; and $r_2r_1p_2$ lifts to a nontrivial cycle in the cover.
The paths $p = p_2p_1$, $q$, $r = r_3r_2r_1$, are drawn in red, blue, and green respectively.}
\label{interiorlemma2}
\end{figure}

Set
$$Q^{\mathcal{S}}_1 := \left\{ a \in Q_1 \, | \, a \not \in x \text{ for all } x \in \mathcal{S} \right\}.$$

\begin{Proposition} \label{min length}
Let $p,q \in Q_{\geq 0}$ be a minimal non-cancellative pair, and let $r$ be a minimal path such that $rp \equiv rq$ (resp.\ $pr \equiv qr$).
If the rightmost (leftmost) arrow subpath of $r^+$ lies in $\mathcal{R}_{p,q}$, then each arrow subpath of $r$ is in $Q_1^{\mathcal{S}}$.
\end{Proposition}

\begin{proof}
Set $\sigma := \prod_{x \in \mathcal{P}} x$, and for a path $s$, set $\overbar{s} := \bar{\eta}(s)$.

Suppose $rp \equiv rq$ and the rightmost arrow subpath of $r^+$ lies in $\mathcal{R}_{p,q}$.
Assume to the contrary that there is a simple matching $x$ that contains an arrow subpath of $r$; then $r$ factors into paths $r = r_3r_2r_1$, where $r_2 \in Q_1 \setminus Q_1^{\mathcal{S}}$ and $r_1,r_3 \in Q_{\geq 0}$.

Since $x$ is simple, there is a path $t'$ from $\operatorname{t}(p)$ to $\operatorname{h}(r_2)$ whose arrow subpaths are not contained in $x$.
By Lemma \ref{p-q in I}, the head of $r_2^+$ lies in the interior of $\mathcal{R}_{p,q}$.
Therefore there is a leftmost nontrivial subpath $t^+$ of $t'^+$ contained in $\mathcal{R}_{p,q}$ with tail on $p^+$ or $q^+$; suppose $t^+$ has tail on $p^+$.
Then $p$ factors into paths $p = p_2p_1$, where $\operatorname{h}(p_1^+) = \operatorname{t}(t^+)$, as shown in Figure \ref{cases}.

Consider the path $s$ such that $(r_2s)^+$ is a unit cycle that lies in $\mathcal{R}_{t,r_2r_1p_2}$.
(Note that $s$ and $r_1p$ may share common subpaths.)
Then $x \nmid \overbar{s}$ since $r_2 \in x$.
Furthermore, $x \nmid \overbar{t}$ since no arrow subpath of $t$ is contained in $x$.
Thus
$$x \nmid \overbar{s} \, \overbar{t} = \overbar{st}.$$
Therefore
\begin{equation} \label{st}
\sigma \nmid \overbar{st}.
\end{equation}

Since $t^+$ lies in $\mathcal{R}_{p,q}$ and $\operatorname{h}(t^+)$ lies in the interior of $\mathcal{R}_{p,q}$, we have proper containment
$$\mathcal{R}_{r_1p_2,st} \subset \mathcal{R}_{p,q}.$$
Furthermore, by (\ref{st}) and Proposition \ref{Bastuff}.1, there is some $m \geq 0$ such that
$$\sigma^m \overbar{st} = \overbar{r_1p_2}.$$
Thus, by the minimality of $p,q$ and Proposition \ref{Bastuff}.2,
$$\sigma^m_{\operatorname{h}(s)}st \equiv r_1p_2.$$
Whence
\begin{equation} \label{stp1}
\sigma^m_{\operatorname{h}(s)}stp_1 \equiv r_1p.
\end{equation}
Therefore
$$\sigma^m \overbar{stp_1} = \overbar{r_1p} = \overbar{r_1} \, \overbar{p} \stackrel{\textsc{(i)}}{=} \overbar{r_1} \, \overbar{q} = \overbar{r_1q},$$
where (\textsc{i}) holds by Proposition \ref{Bastuff}.2.
Furthermore, since $t^+$ lies in $\mathcal{R}_{p,q}$ and $\operatorname{h}(t^+)$ lies in the interior of $\mathcal{R}_{p,q}$, there is proper containment
$$\mathcal{R}_{stp_1,r_1q} \subset \mathcal{R}_{p,q}.$$
Thus, again by the minimality of $p,q$ and Proposition \ref{Bastuff}.2,
\begin{equation} \label{r1q}
\sigma^m_{\operatorname{h}(s)}stp_1 \equiv r_1q.
\end{equation}
Consequently,
$$r_1p \stackrel{\textsc{(i)}}{\equiv} \sigma^m_{\operatorname{h}(s)}stp_1 \stackrel{\textsc{(ii)}}{\equiv} r_1q,$$
where (\textsc{i}) holds by (\ref{stp1}), and (\textsc{ii}) holds by (\ref{r1q}).
But $r_2$ is an arrow.
Therefore $r$ is not minimal.
\end{proof}

\begin{figure}
$$\xy 0;/r.8pc/:
(-6,0)*{\cdot}="1";(-2,3)*{\cdot}="2";(2,3)*{}="4";
(6,0)*{\cdot}="8";
(-2,-3)*{}="5";(2,-3)*{}="7";
(3,0)*{\cdot}="10";
(0,0)*{\cdot}="0";
{\ar@/^/^{}@[red]^{p_1}"1";"2"};{\ar@{-}@/^.1pc/@[red]"2";"4"};{\ar@/^/^{p_2}@[red]"4";"8"};
{\ar@{-}@/_/_{}@[blue]"1";"5"};{\ar@{-}@/_.1pc/_{q}@[blue]"5";"7"};{\ar@/_/_{}@[blue]"7";"8"};
{\ar@[green]^{r_1}"8";"10"};{\ar@{..>}"8";"10"};{\ar@[green]^{r_2}"10";"0"};{\ar@{..>}"10";"0"};
{\ar_t"2";"0"};
{\ar@/^1.7pc/_s"0";"10"};
\endxy$$
\caption{Setup for Proposition \ref{min length}.  
Here, $p_1,p_2, q, r_1, r_3, s \in Q_{\geq 0}$ are paths; $r_2 \in Q_1 \setminus Q_1^{\mathcal{S}}$ is an arrow; and $r_2s$ is a unit cycle.
The paths $p = p_2p_1$, $q$, $r = r_2r_1$, are drawn in red, blue, and green respectively.}
\label{cases}
\end{figure}

\begin{Remark} \label{eats and sleeps and not much more} \rm{
If $p,q \in Q_{\geq 0}$ is a non-cancellative pair that is not minimal, and $r$ is a minimal path satisfying $rp \equiv rq$, then it is possible that no representative of $r^+$ lies in $\mathcal{R}_{p,q}$, and no arrow subpath of $r$ is in $Q_1^{\mathcal{S}}$.
Such an example is given in Figure \ref{s's2}.
}\end{Remark}

\begin{figure}
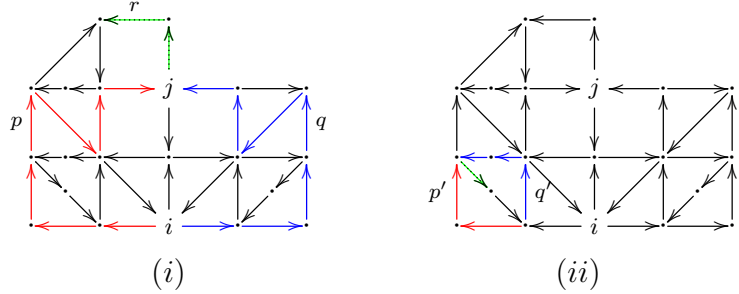

$$\begin{array}{ccc}
\xy 0;/r.31pc/:
(-14,-7)*{\cdot}="1";(-7,-7)*{\cdot}="2";(0,-7)*+{\text{\scriptsize{$i$}}}="3";(7,-7)*{\cdot}="4";(14,-7)*{\cdot}="5";
(-14,0)*{\cdot}="6";(-7,0)*{\cdot}="7";(0,0)*{\cdot}="8";(7,0)*{\cdot}="9";(14,0)*{\cdot}="10";
(-14,7)*{\cdot}="11";(-7,7)*{\cdot}="12";(0,7)*+{\text{\scriptsize{$j$}}}="13";(7,7)*{\cdot}="14";(14,7)*{\cdot}="15";
(-7,14)*{\cdot}="16";(0,14)*{\cdot}="17";
(-10.5,-3.5)*{\cdot}="18";(-10.5,0)*{\cdot}="19";(-10.5,7)*{\cdot}="20";
(10.5,-3.5)*{\cdot}="21";
{\ar@{->}@[red]"2";"1"};{\ar@{->}@[red]"3";"2"};{\ar@{->}@[blue]"3";"4"};{\ar@{->}@[blue]"4";"5"};
{\ar@{->}"7";"19"};{\ar@{->}"19";"6"};{\ar@{->}"8";"7"};{\ar@{->}"8";"9"};{\ar@{->}"9";"10"};
{\ar@{->}"20";"11"};{\ar@{->}"12";"20"};{\ar@{->}@[red]"12";"13"};{\ar@{->}@[blue]"14";"13"};{\ar@{->}"14";"15"};
{\ar@{->}_r@[green]"17";"16"};{\ar@{..>}"17";"16"};
{\ar@{->}@[red]"1";"6"};{\ar@{->}"2";"7"};{\ar@{->}"3";"8"};{\ar@{->}"4";"9"};{\ar@{->}@[blue]"5";"10"};
{\ar@{->}^p@[red]"6";"11"};
{\ar@{->}@[red]"7";"12"};{\ar@{->}"13";"8"};{\ar@{->}@[blue]"9";"14"};
{\ar@{->}_q@[blue]"10";"15"};
{\ar@{->}"16";"12"};{\ar@{->}@[green]"13";"17"};{\ar@{..>}"13";"17"};
{\ar@{->}"6";"18"};{\ar@{->}"18";"2"};{\ar@{->}"7";"3"};{\ar@{->}"9";"3"};{\ar@{->}"10";"21"};{\ar@{->}"21";"4"};
{\ar@{->}@[red]"11";"7"};{\ar@{->}@[blue]"15";"9"};
{\ar@{->}"11";"16"};
\endxy
& \ \ \ \ &
\xy 0;/r.31pc/:
(-14,-7)*{\cdot}="1";(-7,-7)*{\cdot}="2";(0,-7)*+{\text{\scriptsize{$i$}}}="3";(7,-7)*{\cdot}="4";(14,-7)*{\cdot}="5";
(-14,0)*{\cdot}="6";(-7,0)*{\cdot}="7";(0,0)*{\cdot}="8";(7,0)*{\cdot}="9";(14,0)*{\cdot}="10";
(-14,7)*{\cdot}="11";(-7,7)*{\cdot}="12";(0,7)*+{\text{\scriptsize{$j$}}}="13";(7,7)*{\cdot}="14";(14,7)*{\cdot}="15";
(-7,14)*{\cdot}="16";(0,14)*{\cdot}="17";
(-10.5,-3.5)*{\cdot}="18";(-10.5,0)*{\cdot}="19";(-10.5,7)*{\cdot}="20";
(10.5,-3.5)*{\cdot}="21";
{\ar@[red]"2";"1"};{\ar"3";"2"};{\ar"3";"4"};{\ar"4";"5"};
{\ar@[blue]"7";"19"};{\ar@[blue]"19";"6"};{\ar"8";"7"};{\ar"8";"9"};{\ar"9";"10"};
{\ar@{->}"20";"11"};{\ar@{->}"12";"20"};{\ar@{->}"12";"13"};{\ar@{->}"14";"13"};{\ar@{->}"14";"15"};
{\ar@{->}"17";"16"};
{\ar@{->}^{p'}@[red]"1";"6"};{\ar@{->}@[blue]_{q'}"2";"7"};{\ar@{->}"3";"8"};{\ar@{->}"4";"9"};{\ar@{->}"5";"10"};
{\ar@{->}"6";"11"};
{\ar@{->}"7";"12"};{\ar@{->}"13";"8"};{\ar@{->}"9";"14"};
{\ar@{->}"10";"15"};
{\ar@{->}"16";"12"};{\ar@{->}"13";"17"};
{\ar@{->}@[green]"6";"18"};{\ar@{..>}"6";"18"};
{\ar@{->}"18";"2"};{\ar@{->}"7";"3"};{\ar@{->}"9";"3"};{\ar@{->}"10";"21"};{\ar@{->}"21";"4"};
{\ar@{->}"11";"7"};{\ar@{->}"15";"9"};
{\ar@{->}"11";"16"};
\endxy \\
(i) & & (ii)
\end{array}$$
\caption{Setup for Remark \ref{eats and sleeps and not much more}.
In case (i), the paths $p,q$, drawn in red and blue, form a non-cancellative pair that is not minimal.
The green path $r$ is a minimal path satisfying $rp \equiv rq$ and lies outside of $\mathcal{R}_{p,q}$.
In case (ii), the paths $p',q'$, again drawn in red and blue, form a non-cancellative pair that \textit{is} minimal.
The green path $r'$ is a minimal path satisfying $r'p' \equiv r'q'$ and, in contrast to case (i), necessarily lies inside of $\mathcal{R}_{p',q'}$.}
\label{s's2}
\end{figure}

\begin{Corollary} \label{yooohooo}
If $A$ is non-cancellative, then $Q_1^{\mathcal{S}} \not = \emptyset$.
\end{Corollary}

\begin{proof}
For any non-cancellative pair $p,q \in A$, there is some $m \geq 1$ such that $\sigma^m_{\operatorname{h}(p)} p = \sigma^m_{\operatorname{h}(p)} q$.
Furthermore, there is a representative of $\sigma_{\operatorname{h}(p)}$ whose lift lies in $\mathcal{R}_{p,q}$.
The corollary then follows from Proposition \ref{min length}.
\end{proof}

Corollary \ref{yooohooo} is used in \cite{B1} to show that every nondegenerate dimer algebra admits a cyclic contraction.
Our standing assumption that $A$ is nondegenerate thus implies that $A$ admits a cyclic contraction.
\textit{For the remainder of the article, let $\psi: A \to A'$ be a cyclic contraction on a set of arrows $Q_1^*$ unless stated otherwise.}
Set $\sigma := \prod_{x \in \mathcal{S}'}$, and for each $s \in e_jAe_i$, $t \in e_{\ell}A'e_k$, set
\begin{equation} \label{zwb}
\overbar{s} := \bar{\tau}_{\psi}(s) := \bar{\tau} \psi(s), \ \ \ \ \ \overbar{t} := \bar{\tau}(t).
\end{equation}

We will use the following results from \cite{B7}.

\begin{Lemma} \label{Bbstuff}
Let $p,q \in e_iAe_i$ be cycles.
\begin{enumerate}
 \item If $\overbar{p} = \overbar{q}$ and $p-q \in Ze_i$, then
 $$p^2 = pq = qp = q^2.$$
 \item There is an $N \geq 1$ such that for each $n \geq 1$, $p^n \sigma_i^N$ is in $Ze_i$.
 \item If $\overbar{p}$ is in $R$, then there is an $n \geq 1$ such that $p^n$ is in $Ze_i$.
 \item The nilradical of $Z$ consists of the central elements annihilated by $\psi$,
 $$\operatorname{nil}Z = \operatorname{ker}\psi \cap Z.$$
\end{enumerate}
\end{Lemma}

\begin{proof}
The claims are respectively (1): \cite[Proposition 3.4]{B7}; (2), (3): \cite[Proposition 5.4]{B7}; and (4): \cite[Theorem 3.5]{B7}.
\end{proof}

\begin{Theorem} \label{Q* in Qdagger}
If $\psi: A \to A'$ is a cyclic contraction, then no contracted arrow is contained in a simple matching of $A$,
$$Q_1^* \subseteq Q_1^{\mathcal{S}}.$$
\end{Theorem}

\begin{proof}
Assume to the contrary that there is an arrow $\delta \in Q_1^* \setminus Q_1^{\mathcal{S}}$; let $x \in \mathcal{S}$ be a simple matching containing $\delta$.
Let $s \in Q_{\geq 0}$ be a path for which $s \delta$ is a unit cycle.
Since $\delta$ is contracted to a vertex, $\psi(s)$ is a unit cycle in $Q'$.
Whence
\begin{equation} \label{divides s}
\overbar{s} = \bar{\tau}(\psi(s)) = \sigma.
\end{equation}

Since $x$ is simple, there is a cycle $p$ of the subquiver $Q \setminus x$ that passes through each vertex of $Q$ and contains $s$ as a subpath.
Since $s$ is a subpath of $p$, $\psi(s)$ is a subpath of $\psi(p)$.
In particular, $\overbar{s} \mid \overbar{p}$.
Therefore (\ref{divides s}) implies
\begin{equation} \label{sigma divides}
\sigma \mid \overbar{p}.
\end{equation}

Let $u \in \mathbb{Z}^2$ be such that $\psi(p) \in \mathcal{C}'^u$.
Since $A'$ is cancellative, there is a cycle $t \in \hat{\mathcal{C}}'^u_{\operatorname{t}(\psi(p))}$, by Proposition \ref{Bastuff}.5. 
Furthermore, there is an $\ell \in \mathbb{Z}$ such that
$$\overbar{p} = \overbar{\tau} \psi(p) \stackrel{\textsc{(i)}}{=} \bar{\tau}(t) \sigma^{\ell} = \overbar{t} \sigma^{\ell},$$
where (\textsc{i}) holds by Proposition \ref{Bastuff}.1.
But $\sigma \nmid \overbar{t}$, by Proposition \ref{Bastuff}.3.
Whence, $\ell \geq 0$.
Thus, $\ell \geq 1$ by (\ref{sigma divides}).
Therefore
$$\overbar{p} \sigma^{-1} = \overbar{t} \sigma^{\ell-1} =\bar{\tau}(t \sigma_{\operatorname{t}(t)}^{\ell-1}) \in S.$$
In particular, there is a cycle $q$ in $Q$ satisfying
\begin{equation*} \label{q p/sigma}
\overbar{q} = \overbar{p}\sigma^{-1},
\end{equation*}
since the contraction $\psi: A \to A'$ is cyclic.

Set $i := \operatorname{t}(q)$.
Since $p$ contains each vertex in $Q$, the monomial $\overbar{p} = \overbar{q \sigma_i}$ is in $R$, and so we may assume $\operatorname{t}(p) = i$.
Thus there is some $n_1,n_2 \geq 1$ such that the cycles $p^{n_1}$ and $(q \sigma_i)^{n_2}$ are in $Ze_i$, by Lemma \ref{Bbstuff}.3.
Set $n := n_1n_2$.
Then
$$p^n - (q \sigma_i)^n \in Ze_i.$$
This, together with $\overbar{p^n} = \overbar{(q \sigma_i)^n}$, implies
\begin{equation} \label{rp contradiction}
p^{2n} = (q\sigma_i)^{2n},
\end{equation}
by Lemma \ref{Bbstuff}.1.

Finally, let $V$ be an $A$-module with support $Q \setminus x$ and dimension $1^{Q_0}$.
Then $p^{2n}$ does not annihilate $V$ since $p^{2n}$ is cycle of the subquiver $Q \setminus x$.
However, $\sigma_i$ contains an arrow in $x$ since $x$ is a perfect matching.
Thus $(q \sigma_i)^{2n}$ annihilates $V$.
But this contradicts (\ref{rp contradiction}).
\end{proof}

\begin{Lemma} \label{simple matching lemma}
If an arrow annihilates a simple $A$-module of dimension $1^{Q_0}$, then it is contained in a simple matching of $A$.
\end{Lemma}

\begin{proof}
Let $V_{\rho}$ be a simple $A$-module of dimension $1^{Q_0}$, and suppose $\rho(a) = 0$.
Let $i \in Q_0$.
Since $V_{\rho}$ is simple of dimension $1^{Q_0}$, there is a path $p$ from $\operatorname{t}(a)$ to $i$ such that $\rho(p) \not = 0$.
Furthermore, $\sigma_i p = p \sigma_{\operatorname{t}(a)}$ since $\sum_{j \in Q_0}\sigma_j$ is central.
Thus, since $a$ is a subpath $\sigma_{\operatorname{t}(a)}$ (modulo $I$), we have
$$\rho(\sigma_i) \rho(p) = \rho(\sigma_i p ) = \rho(p \sigma_{\operatorname{t}(a)}) = 0.$$
Whence
$$\rho(\sigma_i) = 0.$$
Thus each unit cycle contains at least one arrow that annihilates $V_{\rho}$.
Therefore there are perfect matchings $x_1, \ldots, x_m \in \mathcal{P}$ such that $V_{\rho}$ has support $Q \setminus (x_1 \cup \cdots \cup x_m)$.
Moreover, since $\rho(a) = 0$, there is some $1 \leq \ell \leq m$ such that $x_{\ell}$ contains $a$.

Since $V_{\rho}$ is simple of dimension $1^{Q_0}$, the subquiver $Q \setminus (x_1 \cup \cdots \cup x_m)$ contains a path $r$ that passes through each vertex of $Q$.
But since $r$ is a path in $Q \setminus (x_1 \cup \cdots \cup x_m)$, $r$ is also a path in $Q \setminus x_{\ell}$.
Therefore $x_{\ell}$ is a simple matching containing $a$.
\end{proof}

\begin{Lemma} \label{u = v}
Suppose $A$ is cancellative.
If $p \in \mathcal{C}^u$, $q \in \mathcal{C}^v$, and $\bar{\tau}(p) = \bar{\tau}(q)$, then $u = v$.
\end{Lemma}

\begin{proof}
For a path $s$, set $\overbar{s} := \bar{\tau}(s)$.
Suppose the hypotheses hold, and assume to the contrary that $u \not = v$. 

(i) First suppose $u = mv$ for some $m \in \mathbb{Z}_{\geq 1}$.
Let $s^+$ and $t^+$ be paths in $Q^+$ such that
$$\operatorname{t}(s^+) = \operatorname{t}(p^+), \ \ \ \operatorname{h}(s^+) = \operatorname{t}(q^+), \ \ \ \ \operatorname{t}(t^+) = \operatorname{h}(q^+), \ \ \ \operatorname{h}(t^+) = \operatorname{h}(p^+).$$
Then $ts = \pi(t^+)\pi(s^+)$ is a cycle in $Q$.
Furthermore, there is some $\ell \in \mathbb{Z}$ for which
$$\overbar{t} \overbar{p} \overbar{s} \stackrel{\textsc{(i)}}{=} \overbar{t} \overbar{q} \overbar{s} = \overbar{tqs} \stackrel{\textsc{(ii)}}{=} \overbar{p} \sigma^{\ell},$$
where (\textsc{i}) holds by assumption, and (\textsc{ii}) holds by Proposition \ref{Bastuff}.1.
Whence,
\begin{equation} \label{kza}
\overbar{ts} = \overbar{t} \overbar{s} = \sigma^{\ell}.
\end{equation}
Thus, since $A$ is cancellative, $(ts)^+$ is a cycle in $Q^+$, by Proposition \ref{Bastuff}.6.
Therefore $m = 1$.

(ii) Now suppose $u = mv$ for some $m \in \mathbb{Z}_{\leq 0}$.
Let $s^+$ and $t^+$ be paths in $Q^+$ such that
$$\operatorname{t}(s^+) = \operatorname{h}(p^+), \ \ \ \operatorname{h}(s^+) = \operatorname{t}(q^+), \ \ \ \ \operatorname{t}(t^+) = \operatorname{h}(q^+), \ \ \ \operatorname{h}(t^+) = \operatorname{t}(p^+).$$
As in case (i), $ts$ is a cycle in $Q$ satisfying
\begin{equation} \label{ild}
\overbar{t} \overbar{s} = \overbar{ts} = \sigma^{\ell}
\end{equation}
for some $\ell \geq 0$.
Thus, there is some $n \in \mathbb{Z}$ such that
$$\overbar{p}^2 \sigma^{\ell} \stackrel{\textsc{(i)}}{=} \overbar{p} \overbar{q} \sigma^{\ell}
\stackrel{\textsc{(ii)}}{=} \overbar{p} \overbar{t} \overbar{q} \overbar{s} = \overbar{ptqs} \stackrel{\textsc{(iii)}}{=} \sigma^n,$$
where (\textsc{i}) holds by assumption; (\textsc{ii}) holds by (\ref{ild}); and (\textsc{iii}) holds by Proposition \ref{Bastuff}.1.
Whence, $\overbar{p}^2 = \sigma^{n-\ell}$.
Consequently, $(p^2)^+$, hence $p^+$, is a cycle in $Q^+$, by Proposition \ref{Bastuff}.6.
Similarly, $q^+$ is a cycle in $Q^+$.
Therefore $u = (0,0) = v$.

(iii) Finally, suppose $u$ is not a multiple of $v$. 
Then, since any two non-parallel lines in $\mathbb{R}^2$ intersect, there are lifts of $p$ and $q$ to $Q^+ \subset \mathbb{R}^2$ that intersect at a vertex $i^+ \in Q_0^+$.  
By cyclically permuting $p$ and $q$, we may assume that $\operatorname{t}(p) = \operatorname{t}(q) = i$.

Let $r^+$ be a path from $\operatorname{h}(p^+)$ to $\operatorname{h}(q^+)$; then $r$ is a cycle at $i$.
Since $A$ is cancellative, we may choose $r$ to be in $\hat{\mathcal{C}}_i$, by Proposition \ref{Bastuff}.5.
Whence
\begin{equation} \label{super zephy}
\sigma \nmid \overbar{r},
\end{equation}
by Proposition \ref{Bastuff}.3.
Moreover, there is an $m \in \mathbb{Z}$ such that
$$\overbar{r} \overbar{p} = \overbar{rp} \stackrel{\textsc{(i)}}{=} \overbar{q} \sigma^m \stackrel{\textsc{(ii)}}{=} \overbar{p} \sigma^m,$$
where (\textsc{i}) holds by Proposition \ref{Bastuff}.1, and (\textsc{ii}) holds by assumption.
Thus $\overbar{r} = \sigma^m$.
Hence $r^+$ is a cycle in $Q^+$, by Proposition \ref{Bastuff}.6.
Therefore $\operatorname{h}(p^+) = \operatorname{h}(q^+)$.
But then $u = v$, contrary to assumption. 
\end{proof}

\begin{Lemma} \label{R = S iff}
Let $\psi: A \to A'$ be a cyclic contraction.
Then $R = S$ if and only if $\hat{\mathcal{C}}_i^u \not = \emptyset$ for each $u \in \mathbb{Z}^2$ and $i \in Q_0$.
\end{Lemma}

\begin{proof}
($\Rightarrow$) Suppose $R = S$.
Fix $u \in \mathbb{Z}^2 \setminus 0$ and $i \in Q_0$; we claim that $\hat{\mathcal{C}}^u_i \not = \emptyset$.

Since $A'$ is cancellative, there is a cycle $q \in \hat{\mathcal{C}}'^u$, by Proposition \ref{Bastuff}.5.
In particular, $\sigma \nmid \overbar{q}$ by Proposition \ref{Bastuff}.3.
Since $R = S$, there is a cycle $p \in \mathcal{C}_i$ such that $\overbar{p} = \overbar{q}$.
Furthermore, since $\sigma \nmid \overbar{p}$, $p$ is in $\hat{\mathcal{C}}_i$, by Proposition \ref{Bastuff}.3.
Thus to show that $\hat{\mathcal{C}}^u_i \not = \emptyset$, it suffices to show that $p$ is also in $\mathcal{C}^u$.

Now $\psi(p)$ is in $\mathcal{C}'^u$ since $q$ is in $\mathcal{C}'^u$ and $\overbar{\psi(p)} = \overbar{q}$, by Lemma \ref{u = v}.
Furthermore, $\psi$ cannot contract a cycle in the underlying graph of $Q$ to a vertex \cite[Lemma 3.6]{B2}.
Therefore $p$ is in $\mathcal{C}^u$, proving our claim.

($\Leftarrow$) Conversely, suppose $\hat{\mathcal{C}}_i^u \not = \emptyset$ for each $u \in \mathbb{Z}^2$ and $i \in Q_0$.
To show that $S \subseteq R$, it suffices to show that each monomial in $S$ not divisible by $\sigma$ is in $R$, by Proposition \ref{Bastuff}.8.
So let $g \in S$ be such that $\sigma \nmid g$.
Since $\sigma \nmid g$, there is a cycle $p \in \hat{\mathcal{C}}$ for which $\overbar{p} = g$, by Proposition \ref{Bastuff}.3.
Let $u \in \mathbb{Z}^2$ be such that $p \in \hat{\mathcal{C}}^u$.
If $q$ is another cycle in $\hat{\mathcal{C}}^u$, then $\overbar{q} = \overbar{p}$, by Proposition \ref{Bastuff}.4.
Therefore $g = \overbar{p}$ is in $R$ since $\hat{\mathcal{C}}_i^u \not = \emptyset$ for each $i \in Q_0$.
\end{proof}

\begin{Theorem} \label{cool!}
Let $A = kQ/I$ be a nondegenerate dimer algebra.
The following are equivalent:
\begin{enumerate}
 \item $A$ is cancellative.
 \item Each arrow annihilates a simple $A$-module of dimension $1^{Q_0}$.
 \item Each arrow is contained in a simple matching, $Q_1^{\mathcal{S}} = \emptyset$.
 \item The center $R$ of $\Lambda$ equals the cycle algebra $S$.
 \item If $\psi: A \to A'$ is a cyclic contraction, then $\psi$ is trivial, $Q_1^* = \emptyset$.
\end{enumerate}
\end{Theorem}

\begin{proof}
(2) $\Leftrightarrow$ (3): Lemma \ref{simple matching lemma}.

(3) $\Rightarrow$ (1): Proposition \ref{min length}.

(4) $\Rightarrow$ (3): Suppose $R = S$ (recall that $S$ is independent of the choice of cyclic contraction).
Then $\hat{\mathcal{C}}^u_i \not = \emptyset$ for each $i \in Q_0$ and $u \in \mathbb{Z}^2$, by Lemma \ref{R = S iff}.
Therefore $Q_1^{\mathcal{S}} = \emptyset$ by \cite[Theorem 4.25]{B2}.

(1) $\Rightarrow$ (4): If $A$ is cancellative, then $R = S$ by Proposition \ref{Bastuff}.7.

(1) $\Rightarrow$ (5): Suppose $A$ is cancellative, and $\psi: A \to A'$ is a cyclic contraction.
Then
$$Q_1^* \stackrel{\textsc{(i)}}{\subseteq} Q_1^{\mathcal{S}} \stackrel{\textsc{(ii)}}{=} \emptyset,$$
where (\textsc{i)} holds by Theorem \ref{Q* in Qdagger}, and (\textsc{ii)} holds by the implication (1) $\Rightarrow$ (2).
Therefore $Q_1^* = \emptyset$.

(5) $\Rightarrow$ (4): Clear.
\end{proof}

We will use the following definition in the proof of Proposition \ref{homotopy not noetherian}.

\begin{Definition} \label{impression definition} \rm{\cite[Definition 2.6]{B2}
Let $A$ be a finitely generated $k$-algebra and let $Z$ be its center.
An \textit{impression} of $A$ is an algebra monomorphism $\tau: A \hookrightarrow M_d(B)$ to a matrix ring over a commutative finitely generated $k$-algebra $B$, such that
\begin{itemize}
 \item for generic $\mathfrak{b} \in \operatorname{Max}B$, the composition
\begin{equation*} \label{composition}
A \stackrel{\tau}{\longrightarrow} M_d(B) \stackrel{1}{\longrightarrow} M_d\left(B/\mathfrak{b} \right) \cong M_d(k)
\end{equation*}
is surjective; and
 \item the morphism $\operatorname{Max}B \rightarrow \operatorname{Max}\tau(Z)$, $\mathfrak{b} \mapsto \mathfrak{b} \cap \tau(Z)$, is surjective.
\end{itemize}
} \end{Definition}

\begin{Proposition} \label{homotopy not noetherian}
Suppose $A$ is non-cancellative.
The ghor algebra $\Lambda$ is nonnoetherian and an infinitely generated module over its center $R$.
\end{Proposition}

\begin{proof}
Consider the $k$-linear map
$$\tau_{\psi}:A \to M_{|Q_0|}(k[\mathcal{S}'])$$ 
defined for each $i,j \in Q_0$ and $p \in e_jA e_i$ by
$$p \mapsto \bar{\tau}\psi(p) e_{ji}.$$
It is shown in \cite[Theorem 5.9.1]{B2} that this map induces an impression of $\Lambda$,
\begin{equation} \label{ghmily}
\tau_{\psi}: \Lambda \to M_{|Q_0|}(k[\mathcal{S}']).
\end{equation}
Furthermore, it is shown in \cite[Theorem 3.18]{B6} that on the algebraic variety $\operatorname{Max}S$, the following loci coincide:
 $$U^*_{S/R} : = \left\{ \mathfrak{n} \in \operatorname{Max}S \, | \, R_{\mathfrak{n} \cap R} \text{ is noetherian} \right\} = \left\{ \mathfrak{n} \in \operatorname{Max}S \, | \, R_{\mathfrak{n} \cap R} = S_{\mathfrak{n}} \right\} =: U_{S/R}.$$
But $R \not = S$, by Theorem \ref{cool!}.
Therefore $\Lambda$ is nonnoetherian and an infinitely generated $R$-module by \cite[Theorem 4.1.2]{B4}.\footnote{The assumption that $k$ is uncountable in \cite[Theorem 4.1]{B4} is only used in 4.1.1, and not in 4.1.2.}
\end{proof}

Recall (\ref{zwb}): for $p \in e_jAe_i$, $i,j \in Q_0$, set $\bar{\tau}_{\psi}(p) := \bar{\tau}\psi(p)$.

\begin{Proposition} \label{s^n}
Let $p \in e_iAe_i$ be a cycle.
If $\overbar{p} \not \in \bar{\tau}_{\psi}(e_jAe_j)$ and $\sigma \nmid \overbar{p}$, then for each $n \geq 1$,
$$\overbar{p}^n \not \in \bar{\tau}_{\psi}(e_jAe_j).$$
Consequently, if $\overbar{p} \not \in R$ and $\sigma \nmid \overbar{p}$, then for each $n \geq 1$, $\overbar{p}^n \not \in R$.
Moreover, if $A$ is non-cancellative, then such a cycle exists.
\end{Proposition}

\begin{proof}
(i) Assume to the contrary that there is a cycle $p \in e_ikQe_i$ such that $\overbar{p} \not \in R$, $\sigma \nmid \overbar{p}$, and $\overbar{p}^n \in R$ for some $n \geq 2$.
Let $u \in \mathbb{Z}^2$ be such that $p \in \mathcal{C}^u_i$.
Since $\overbar{p}$ is not in $R$, there is a vertex $j \in Q_0$ such that
\begin{equation} \label{boltzmann brain}
\overbar{p} \not \in \bar{\tau}_{\psi}\left( e_jAe_j \right).
\end{equation}
Furthermore, since $\overbar{p}^n$ is in $R$, $p^n$ homotopes to a cycle $q \in e_ikQe_i$ that passes through $j$,
\begin{equation} \label{q = pn}
q \equiv p^n.
\end{equation}

For $v \in \mathbb{Z}^2$, denote by $q^+_v \in \pi^{-1}(q)$ the preimage with tail
$$\operatorname{t}(q^+_v ) = \operatorname{t}(q^+) + v \in Q_0^+.$$
Since $Q^+$ embeds in the plane $\mathbb{R}^2$, there is a path $r^+$ from $j^+$ to $j^+ + u$ that is constructed from subpaths of $q^+$, $q^+_u$, and $q^+_{mu}$, for some $m \in \mathbb{Z}$; see Figure \ref{proposition figure3}.
In particular, the cycle $r := \pi(r^+) \in e_jkQe_j$ is in $\mathcal{C}^u_j$.

Since $\sigma \nmid \overbar{p}$, there is a simple matching $x$ such that $x \nmid \overbar{p}$.
Whence $x \nmid \overbar{q}$ by (\ref{q = pn}).
Thus $x \nmid \overbar{r}$.
We therefore have
\begin{equation} \label{flux}
\sigma \nmid \overbar{p}, \ \ \ \sigma \nmid \overbar{r}, \ \ \text{ and } \ \ p,r \in \mathcal{C}^u.
\end{equation}
But if $s,t$ are cycles in $\mathcal{C}^u$, then there is an $\ell \in \mathbb{Z}$ such that $\overbar{s} = \overbar{t}\sigma^{\ell}$ \cite[Lemma 4.19]{B2}.
Consequently, (\ref{flux}) implies $\overbar{r} = \overbar{p}$.
Therefore
$$\overbar{p} = \overbar{r} \in \bar{\tau}_{\psi}\left( e_jAe_j \right),$$
contrary to (\ref{boltzmann brain}).

(ii) Now suppose no cycle $p$ exists for which $\overbar{p}^n \not \in R$ for each $n \geq 1$.
Then by Claim (i), if $q$ is a cycle satisfying $\overbar{q} \not \in R$, then $\sigma \mid \overbar{q}$.
By the contrapositive of this assumption, for each cycle $q$ satisfying $\sigma \nmid \overbar{q}$, we have $\overbar{q} \in R$.
But $S$ is generated by $\sigma$ and a set of monomials in $k[\mathcal{S}']$ not divisible by $\sigma$, by Proposition \ref{Bastuff}.8.
Whence $S \subseteq R$ since $\sigma \in R$.
Thus $S = R$.
Therefore $A$ is cancellative by Theorem \ref{cool!}.
\end{proof}

\begin{figure}
$$\xy 0;/r.3pc/:
(-30,-14)*+{\text{\scriptsize{$i$}}}="1";
(-20,-14)*+{\text{\scriptsize{$i$}}}="2";
(-10,-14)*+{\text{\scriptsize{$i$}}}="3";
(10,-14)*+{\text{\scriptsize{$i$}}}="4";
(20,-14)*+{\text{\scriptsize{$i$}}}="5";
(30,-14)*+{\text{\scriptsize{$i$}}}="6";
(-15,7)*{}="7";
(-5,7)*{}="8";
(5,7)*{}="9";
(15,7)*{}="10";
(-2.5,3.5)*{\cdot}="11";
(0,0)*{}="12";
(7.5,3.5)*{\cdot}="13";
(17.5,3.5)*{\cdot}="14";
{\ar@{-}^{p^n = q}"1";"7"};{\ar"8";"11"};{\ar"11";"4"};{\ar@{-}@/^/"7";"8"};
{\ar@{-}^q"2";"8"};{\ar"9";"13"};{\ar"13";"5"};{\ar@{-}@/^/"8";"9"};
{\ar@{-}^q"3";"9"};{\ar"10";"14"};{\ar"14";"6"};{\ar@{-}@/^/"9";"10"};
{\ar_p"1";"2"};{\ar_p"2";"3"};{\ar_p"4";"5"};{\ar_p"5";"6"};
{\ar@[red]"11";"12"};{\ar@[red]_r"12";"9"};{\ar@[red]"9";"13"};
{\ar@{}^j"11";"11"};{\ar@{}^j"13";"13"};{\ar@{}^j"14";"14"};
\endxy$$
\caption{Setup for Proposition \ref{s^n}.  The paths are in the cover $Q^+$ (with the superscripts $^+$ omitted).  
The path $r^+$ is drawn in red, and its projection $r = \pi(r^+)$ to $Q$ is a cycle at $j$.}
\label{proposition figure3}
\end{figure}

\begin{Lemma} \label{great scott}
For each $i \in Q_0$, there is an inclusion $\bar{\tau}_{\psi}(Ze_i) \subseteq R$.
\end{Lemma}

\begin{proof}
For $i \in Q_0$, we have
$$\bar{\tau}_{\psi}(Ze_i) = \bar{\tau} \psi \left( \left( Z/ (\operatorname{ker}\psi \cap Z) \right)e_i \right) \stackrel{\textsc{(i)}}{=} \bar{\tau}_{\psi} \left( \left( Z/\operatorname{nil}Z \right)e_i \right) \stackrel{\textsc{(ii)}}{\subseteq} R,$$
where (\textsc{i}) holds by Lemma \ref{Bbstuff}.4, and (\textsc{ii}) holds by \cite[Theorem 4.1]{B7}.
\end{proof}

\begin{Theorem} \label{nonnoetherian theorem}
Suppose $A$ is non-cancellative.
Then $A$, its center $Z$, its reduced center $\hat{Z} := Z/\operatorname{nil}Z$, and its ghor center $R$, are all nonnoetherian algebras.
\end{Theorem}

\begin{proof}
(i) We first claim that $A$ is nonnoetherian.
Indeed, since $A$ is non-cancellative, there is a cycle $p \in e_iAe_i$ and vertex $j \in Q_0$ such that for each $n \geq 1$,
\begin{equation} \label{not in R}
\overbar{p}^n \not \in \bar{\tau}_{\psi}(e_jAe_j),
\end{equation}
by Proposition \ref{s^n}.
Furthermore, there is an $N \geq 1$ such that for each $n \geq 1$,
\begin{equation*} \label{p^m sigma}
\overbar{p}^n \sigma^N \in \bar{\tau}_{\psi}(Ze_j) \subseteq \bar{\tau}_{\psi}(e_jAe_j).
\end{equation*}
by Lemma \ref{Bbstuff}.2.
Let $q_n \in Ze_j$ be such that
$$\overbar{q}_n := \overbar{p}^n \sigma^N.$$

Consider the ascending chain of (two-sided) ideals of $A$,
$$\left\langle q_1 \right\rangle \subseteq \left\langle q_1,q_2 \right\rangle \subseteq \left\langle q_1,q_2,q_3 \right\rangle \subseteq \cdots.$$
Assume to the contrary that the chain stabilizes; then for some $m \geq 1$ there are elements $a_1, \ldots, a_m \in A$ satisfying
\begin{equation} \label{Doc!}
q_m = \sum_{n = 1}^{m-1} a_nq_n.
\end{equation}
Since each $q_n$ is in $e_jAe_j$, we have
$$q_m = e_jq_m = e_j \sum_{n<m} a_n e_j q_n = \sum_{n<m} (e_j a_n e_j) q_n.$$
In particular, we may take each $a_n$ to be in $e_jAe_j$.
Furthermore, (\ref{Doc!}) implies
$$\overbar{p}^m \sigma^N = \overbar{q}_m = \sum_{n<m} \overbar{a_n q_n} = \sum_{n<m} \overbar{a}_n\overbar{q}_n = \sum_{n<m} \overbar{a}_n \overbar{p}^n \sigma^N.$$
Thus, since $k[\mathcal{S}']$ is an integral domain, we have
\begin{equation} \label{p^m}
\overbar{p}^m = \sum_{n<m} \overbar{a}_n \overbar{p}^n,
\end{equation}
with each $\overbar{a}_n$ in $\bar{\tau}_{\psi}(e_jAe_j)$.

Now the subalgebra $\bar{\tau}_{\psi}(e_jAe_j)$ is generated by a set of monomials in the polynomial ring $k[\mathcal{S}']$.
Thus each polynomial $\overbar{a}_n$ is a $k^{\times}$-linear combination of distinct monomials $\alpha_{n,\ell}$ in $\bar{\tau}_{\psi}(e_jAe_j)$.
Furthermore, (\ref{p^m}) implies that for some $0 \leq n < m$, there is an $\ell$ for which $\alpha_{n,\ell} = \overbar{p}^{m-n}$.
But then $\overbar{p}^{m-n}$ is in $\bar{\tau}_{\psi}(e_jAe_j)$, contrary to (\ref{not in R}).
Therefore $A$ is nonnoetherian.

(ii) We claim that $Z$ is nonnoetherian.
The cycles $q_n$ were chosen to be in $Ze_j$, so we may consider the chain of ideals of $Ze_j$,
$$(q_1)Z \subseteq (q_1,q_2)Z \subseteq (q_1,q_2,q_3)Z \subseteq \cdots.$$
This chain does not stabilize by the argument in Claim (i) since $Ze_j \subseteq e_jAe_j$.
Thus $Ze_j$ is nonnoetherian.
Consequently, $Z$ is nonnoetherian since the vertex idempotents are orthogonal.

(iii) We claim that $R$ is nonnoetherian.
By Lemma \ref{great scott}, the monomials $\overbar{q}_n$ are in $R$.
Thus we may consider the chain of ideals of $R$,
$$(\overbar{q}_1)R \subseteq (\overbar{q}_1,\overbar{q}_2)R \subseteq (\overbar{q}_1,\overbar{q}_2,\overbar{q}_3)R \subseteq \cdots.$$
This chain does not stabilize by the argument in Claim (i) since $R \subseteq \bar{\tau}_{\psi}(e_jAe_j)$.
Therefore $R$ is nonnoetherian.

(iv) Finally, we claim that the reduced center $\hat{Z}$ is nonnoetherian.
There are no cycles in $(\operatorname{nil}Z)e_j$ since $\operatorname{nil}Z = \operatorname{ker}\psi \cap Z$, by Lemma \ref{Bbstuff}.4.
In particular, the cycles $q_n \in Ze_j$ are not in $(\operatorname{nil}Z)e_j$, and so are nonzero in $\hat{Z}e_j$.
The claim then follows from Claim (ii).
\end{proof}

Although $Z$, $\hat{Z}$, and $R$ are nonnoetherian, it is shown in \cite[Theorem 1.1]{B6} that they each have Krull dimension 3 and are generically noetherian.
In particular, though they do not satisfy the ascending chain condition on all ideals, they do satisfy the ascending chain condition on prime ideals.

\begin{Lemma} \label{S over R}
If $A$ is a finitely generated $Z$-module, then $S$ is a finitely generated $R$-module.
\end{Lemma}

\begin{proof}
Suppose $A$ is finite over $Z$.
Then there are elements $a_1, \ldots, a_N \in A$ such that
$$A = \sum_{n = 1}^N Z a_n.$$
Whence, for each $i \in Q_0$,
\begin{multline*}
\bar{\tau}_{\psi}\left( e_iAe_i \right) = \bar{\tau}_{\psi} \left( e_i \sum_n Z a_n e_i \right) = \bar{\tau}_{\psi}\left( \sum_n Ze_i (e_i a_n e_i) \right) \\ = \sum_n \bar{\tau}_{\psi}(Ze_i) \overbar{e_ia_n e_i} \stackrel{\textsc{(i)}}{\subseteq} \sum_n R \, \overbar{e_i a_n e_i},
\end{multline*}
where (\textsc{i}) holds by Lemma \ref{great scott}.
Thus $\bar{\tau}_{\psi}\left( e_iAe_i \right)$ is finite over $R$.
But then
\begin{align*}
S & = k\left\langle \prod_{i \in Q_0} \bar{\tau}_{\psi}(e_iAe_i) \right\rangle \\
& = R\left\langle \prod_{i \in Q_0} \bar{\tau}_{\psi}(e_iAe_i) \right\rangle \\
& = R\left\langle \prod_{i \in Q_0} \overbar{e_ia_{n(i)} e_i} \ \mid \ 1 \leq n(i) \leq N \right\rangle.
\end{align*}
Therefore $S$ is finite over $R$.
\end{proof}

\begin{Theorem} \label{non-finitely generated}
Suppose $A$ is non-cancellative.
Then $A$ is an infinitely generated $Z$-module.
\end{Theorem}

\begin{proof}
The cycle algebra $S$ is a finitely generated $k$-algebra by Proposition \ref{Bastuff}.8, whereas $R$ is an infinitely generated $k$-algebra by Theorem \ref{nonnoetherian theorem}.
Thus $S$ is infinite over $R$ by the Artin-Tate lemma.
Therefore $A$ is infinite over $Z$ by Lemma \ref{S over R}.
\end{proof}

An immediate question following Theorem \ref{non-finitely generated} is whether non-cancellative dimer algebras satisfy a polynomial identity.
We note that ghor algebras are always PI since they are subalgebras of matrix rings over commutative rings \cite[Introduction]{B2}.

\begin{Proposition} \label{inout}
Suppose $\psi: A \to A'$ is a cyclic contraction.
If the head (or tail) of each arrow in $Q_1^*$ has indegree $1$, then $A$ contains a free subalgebra.
In particular, $A$ is not PI.
\end{Proposition}

\begin{proof}
Let $a \in Q_1^*$; then the indegree of $\operatorname{h}(a)$ is $1$ by assumption.
It suffices to suppose the indegree of $\operatorname{t}(a)$ is at least $2$.
Consider the paths $p,q$ and the path $b$ of maximal length such that $bap$ and $baq$ are unit cycles.
Let $b' \in Q'_1$ be a leftmost arrow subpath of $b$.
Since $b$ has maximal length, the indegree of $\operatorname{h}(b') = \operatorname{h}(b)$ is at least $2$.
In particular, $b'$ is not contracted.
Furthermore, no vertex in $Q'$ has indegree $1$ since $A'$ is cancellative.
Whence
$$\psi(bap) = b'p \ \ \ \text{ and } \ \ \ \psi(baq) = b'q$$
are unit cycles in $Q'$.
Therefore $p,q$ is a non-cancellative pair and $\overbar{p} = \overbar{q}$.

Since $A'$ is cancellative, there is a simple matching $x \in \mathcal{S}'$ which contains $b'$, by Theorem \ref{cool!}.
Furthermore, since $x$ is a simple matching, there is a path $s$ in $Q'$ from $\operatorname{h}(p)$ to $\operatorname{t}(p)$ which is a path of $Q' \setminus x$.
In particular,
$$x \nmid \bar{\tau}(sp) = \bar{\tau}(sq).$$

Since the indegree of the head of each contracted arrow is 1, $\psi$ is surjective.
Thus there is path $r$ in $Q$ from $\operatorname{h}(p)$ to $\operatorname{t}(p)$ satisfying $\psi(r) = s$.
Whence
$$x \nmid \bar{\tau}(sp) = \bar{\tau}_{\psi}(rp) = \bar{\tau}_{\psi}(rq).$$
But then $b$ is not a subpath of $rp$ or $rq$ (modulo $I$) since $x \mid \bar{\tau}_{\psi}(b')$.
Consequently, there are no relations between the cycles $rp$ and $rq$.
Therefore
$$k\left\langle rp, rq \right\rangle$$
is a free subalgebra of $A$.
\end{proof}

\begin{Remark} \rm{
An example of a (nondegenerate) dimer algebra, with no vertex of indegree $1$ and with a free subalgebra, is given in Figure \ref{example}.
Indeed, let $a,b$ be the red and blue arrows in $Q$.
Then $k\left\langle a,b \right\rangle$ is a free subalgebra of $A$, and so $A$ is not PI.
}\end{Remark}

\begin{Theorem} \label{iff}
Let $A$ be a non-degenerate dimer algebra.
The following are equivalent:
\begin{enumerate}
 \item $A$ is cancellative.
 \item $A$ is noetherian.
 \item $Z$ is noetherian.
 \item $A$ is a finitely generated $Z$-module.
 \item The vertex corner rings $e_iAe_i$ are pairwise isomorphic algebras.
 \item Each vertex corner ring $e_iAe_i$ is isomorphic to $Z$.
\end{enumerate}
Furthermore, each condition is equivalent to each of the conditions (2) -- (6) with $A$ and $Z$ replaced by $\Lambda$ and its center $R$.
\end{Theorem}

\begin{proof}
First suppose $A$ is cancellative.
Then $A$ is finite over $Z$ and $Z$ is noetherian, by Propositions \ref{Bastuff}.8.
Therefore $A$ is noetherian.
Furthermore, the vertex corner rings are pairwise isomorphic and isomorphic to $Z$, by Propositions \ref{Bastuff}.7 and \ref{Bastuff}.8.
Finally, since $A$ is cancellative, we have $A = \Lambda$.

Conversely, suppose $A$ is non-cancellative.
Then $A$, $Z$, and $R$ are nonnoetherian by Theorem \ref{nonnoetherian theorem}; $A$ is infinite over $Z$ by Theorem \ref{non-finitely generated}; $\Lambda$ is nonnoetherian and infinite over $R$ by Proposition \ref{homotopy not noetherian}; and the vertex corner rings are not all isomorphic by Proposition \ref{s^n}.
\end{proof}

\ \\
\textbf{Acknowledgments.}
The author would like to thank an anonymous referee and Gwyn Bellamy for their helpful comments and suggestions.
The author was supported by the Austrian Science Fund (FWF) grant P 30549-N26.
Part of this article was completed while the author was a research fellow at the Heilbronn Institute for Mathematical Research at the University of Bristol.

\bibliographystyle{hep}
\def\cprime{$'$} \def\cprime{$'$}

\end{document}